\documentclass[12pt, reqno]{amsart}
\makeatletter
\@namedef{subjclassname@1991}{$\mathrm{1991}$ Mathematics Subject Classification}
\@namedef{subjclassname@2000}{$\mathrm{2000}$ Mathematics Subject Classification}
\@namedef{subjclassname@2010}{$\mathrm{2010}$ Mathematics Subject Classification}
\@namedef{subjclassname@2020}{$\mathrm{2020}$ Mathematics Subject Classification}
\makeatother
\usepackage{amsmath,amsthm, amscd, amsfonts, amssymb, graphicx, color}
\usepackage[bookmarksnumbered, colorlinks, plainpages,linkcolor=blue,urlcolor=blue,citecolor=blue]{hyperref}
\newtheorem{thm}{Theorem}[section]
\newtheorem{cor}[thm]{Corollary}
\newtheorem{lem}[thm]{Lemma}

\newtheorem{defn}[thm]{Definition}

\numberwithin{equation}{section}

\begin{document}

\title{On generalized core-EP invertibility in a Banach algebra}

\author{Huanyin Chen}
\author{Marjan Sheibani$^*$}
\address{School of Big Data, Fuzhou University of International Studies and Trade, Fuzhou 350202, China}
\email{<huanyinchenfz@163.com>}
\address{Farzanegan Campus, Semnan University, Semnan, Iran}
\email{<m.sheibani@semnan.ac.ir>}

\subjclass[2020]{16U90, 15A09, 46H05.} \keywords{generalized core-EP inverse; generalized Drazin inverse; reverse order; generalized core-EP
order, complex matrix.}

\begin{abstract} We present new properties of generalized core-EP inverse in a Banach *-algebra. We characterize this new generalized inverse by using involved annihilators. The generalized core-EP inverse for products is obtained. The core-EP orders for Banach *-algebra elements are thereby investigated. As applications, new properties of the core-EP inverse for block complex matrices are given.
\end{abstract}

\maketitle

\section{Introduction}

A Banach algebra is called a Banach *-algebra if there exists an involution $*: x\to x^*$ satisfying $(x+y)^*=x^*+y^*, (\lambda x)^*=\overline{\lambda} x^*, (xy)^*=y^*x^*, (x^*)^*=x$. Let $\mathcal{A}$ be a Banach *-algebra. An element $a\in \mathcal{A}$ has core-EP inverse (i.e., pseudo core inverse) if there exist $x\in \mathcal{A}$ and $k\in \Bbb{N}$ such that $$xa^{k+1}=a^k, ax^2=x, (ax)^*=ax.$$ If such $x$
exists, it is unique, and denote it by $a^{\tiny\textcircled{D}}$. The core-EP invertibility in a Banach *-algebra is attractive. This notion was
introduced by Gao and Chen in 2018 (see~\cite{{GC}}). This is a natural extension of the core inverse which is the first studied by Baksalary and Trenkler
for a complex matrix in 2010 (see~\cite{BT}). Rakic et al. (see~\cite{R}) generalized the core inverse of a complex matrix to the case of an element in a
ring. An element $a$ in a Banach *-algebra $\mathcal{A}$ has core inverse if and only if there exist $x\in \mathcal{A}$ such that $$a=axa,
x\mathcal{A}=a\mathcal{A}, \mathcal{A}x=\mathcal{A}a^*.$$ If such $x$ exists, it is unique, and denote it by $a^{\tiny\textcircled{\#}}$. Recently, many
authors have studied core and core-EP inverses from many different views, e.g., ~\cite{CZ2,GC2,GC3,K,M,R,XS,XS1,Z2}.

Recall that $a\in \mathcal{A}$ has g-Drazin inverse (i.e., generalized Drazin inverse) if there exists $x\in \mathcal{A}$ such that
$ax^2=x, ax=xa, a-a^2x\in \mathcal{A}^{qnil}.$ Such $x$ is unique, if exists, and denote it by $a^d$. Here, $\mathcal{A}^{qnil}=\{a\in \mathcal{A}~\mid~1+\lambda a\in \mathcal{A}^{-1}\}.$ As is well known, we have
$a\in \mathcal{A}^{qnil}\Leftrightarrow \lim\limits_{n\to \infty}\parallel a^n\parallel^{\frac{1}{n}}=0.$ We use $\mathcal{A}^d$ to stand for the set of all generalized Drazin invertible element $a$ in $\mathcal{A}$.

Let $\mathcal{B}(X)$ be the algebra of bounded linear operators over a Hilbert space $X$. In~\cite{MD3}, Mosi\'c and Djordjevi\'c introduced and studied core-EP
inverse for a operator $T\in \mathcal{B}(X)^d$. This new generalized inverse was extensively investigated in ~\cite{MD,MD3,MD4}. Recently, Mosi\'c extended core-EP inverse of bounded linear operators on Hilbert spaces to elements of a $C^*$-algebra by means of range projections (see~\cite{MDD}). 

These research mentioned above raise our unifying them and introduce a new generalized inverse.

\begin{defn} An element $a\in \mathcal{A}$ has generalized core-EP inverse if there exists $x\in \mathcal{A}$ such that $$x=ax^2, (ax)^*=ax, \lim_{n\to \infty}||a^n-xa^{n+1}||^{\frac{1}{n}}=0.$$ \end{defn}

From authors' recent works on weighted generalized core inverse (see~\cite{CM1}), the preceding $x$ is unique if it exists, and denoted by $a^{\tiny\textcircled{d}}$. We use $\mathcal{A}^{\tiny\textcircled{d}}$ to stand for the set of all generalized core-EP invertible element $a$ in $\mathcal{A}$. Here we list main results about generalized core-EP inverse by choosing the weight $e=1$ in~\cite{CM1,CM2}. 

\begin{thm} Let $\mathcal{A}$ be a Banach *-algebra, and let $a\in \mathcal{A}$. Then the following are equivalent:\end{thm}
\begin{enumerate}
\item [(1)] $a\in \mathcal{A}^{\tiny\textcircled{d}}$.
\item [(2)] There exist $x,y\in \mathcal{A}$ such that $$a=x+y, x^*y=yx=0, x\in
\mathcal{A}^{\tiny\textcircled{\#}}, y\in \mathcal{A}^{qnil}.$$
\item [(3)] There exists a projection $p\in \mathcal{A}$ (i.e., $p=p^2=p^*$)
 such that $$a+p\in \mathcal{A}^{-1}, pa=pap\in \mathcal{A}^{qnil}.$$
\item [(4)] $xax=x, im(x)=im(x^*)=im(a^d)$.
\vspace{-.5mm}
\item [(5)] $a\in \mathcal{A}^d$ and there exists a projection $q\in \mathcal{A}$ such that $a^d\mathcal{A}=q\mathcal{A}$. In this case, $a^{\tiny\textcircled{d}}=a^dq.$
\item [(6)] $a\in \mathcal{A}^d$ and $a^d\in \mathcal{A}^{(1,3)}$. In this case, $a^{\tiny\textcircled{d}}=(a^d)^2(a^d)^{(1,3)}.$
\item [(7)] $a\in \mathcal{A}^d$ and $a^d\in \mathcal{A}^{\tiny\textcircled{\#}}$. In this case, $a^{\tiny\textcircled{d}}=(a^d)^2(a^d)^{\tiny\textcircled{\#}}.$
\item [(8)] $a\in \mathcal{A}^d$ and $aa^d\in \mathcal{A}^{(1,3)}$. In this case, $a^{\tiny\textcircled{d}}=a^d(aa^d)^{(1,3)}.$
\end{enumerate}

In Section 2, we characterize generalized core-EP inverse for an element in a Banach *-algebra by using involved annihilators. We prove that 
$a^{\tiny\textcircled{d}}=x$ if and only if $xax=x, \ell(x)=\ell(x^*)=\ell(a^d)$ if and only if there exists a projection $q\in \mathcal{A}$ such that $\ell{(a^d)}=\ell{(q)}.$ 

In Section 3, we establish the reverse order law for generalized core-EP inverses. Let $a,b\in \mathcal{A}^{\tiny\textcircled{d}}$. If $ab^2=b^2a=bab$ and $a^*b^2=b^2a^*=ba^*b$, we prove that $ab\in \mathcal{A}^{\tiny\textcircled{d}}$ and $(ab)^{\tiny\textcircled{d}}=a^{\tiny\textcircled{d}}b^{\tiny\textcircled{d}}=b^{\tiny\textcircled{d}}a^{\tiny\textcircled{d}}.$ 

Finally, in Section 4, the generalized core EP-order for
Banach *-algebra elements are introduced. Let $a,b\in \mathcal{A}^{\tiny\textcircled{d}}$. We define $a\leq^{{\tiny\textcircled{d}}}b$ if
$aa^{{\tiny\textcircled{d}}}=ba^{{\tiny\textcircled{d}}}$ and
$a^{{\tiny\textcircled{d}}}a=a^{{\tiny\textcircled{d}}}b$. The characterizations of the generalized core EP-order are
present. As applications, new properties of the core-EP order for block complex matrices are obtained.

Throughout the paper, all Banach *-algebras are complex with an identity. Let ${\Bbb C}^{n\times n}$ be the *-Banach algebra of 
all $n\times n$ complex matrices with conjugate transpose $*$. $\mathcal{A}^{-1}, \mathcal{A}^{D},  \mathcal{A}^{\tiny\textcircled{\#}}$ and $\mathcal{A}^{\tiny\textcircled{D}}$ denote the sets of
all invertible, Drazin invertible, core invertible and core-EP invertible in $\mathcal{A}$, respectively. Let $a\in \mathcal{A}^d$. We
use $a^{\pi}$ to stand for the spectral idempotent of $a$ corresponding to $\{ 0\}$, i.e., $a^{\pi}=1-aa^d$. ${\ell}(a)$ and $r(a)$ denote the left and right annihilators of $a$ respectively. 

\section{characterizations by using involved annihilators}

The aim of this section is to characterize generalized core-EP inverse of a Banach *-algebra element by using annihilators. We now derive

\begin{thm} Let $a\in \mathcal{A}^d$. Then\end{thm}
\begin{enumerate}
\item [(1)] $a\in \mathcal{A}^{\tiny\textcircled{d}}$ and $a^{\tiny\textcircled{d}}=x$.
\vspace{-.5mm}
\item [(2)] $xax=x, (ax)^*=ax, (xa-1)a^d=0$ and $im(x)\subseteq im(a^d)$.
\vspace{-.5mm}
\item [(3)] $xax=x, \ell(x)=\ell(x^*)=\ell(a^d)$.
\end{enumerate}
\begin{proof} $(1)\Rightarrow (2)$ Let $a^{\tiny\textcircled{d}}=x$. Then $(ax)^*=ax$. In view of Theorem 1.2, we have $xax=x$. Moreover, we check that $$\begin{array}{rll}
||aa^d-xa^2a^d||^{\frac{1}{n}}&=&||a^n(a^d)^n-xa^{n+1}(a^d)^n||^{\frac{1}{n}}\\
&=&||a^n-xa^{n+1}||^{\frac{1}{n}}||a^d||.
\end{array}$$ Since $\lim_{n\to \infty}||a^n-xa^{n+1}||^{\frac{1}{n}}=0,$ we have
$$\lim_{n\to \infty}||aa^d-xa^2a^d||^{\frac{1}{n}}=0.$$ This implies that $aa^d=xa^2a^d$, and so $(xa-1)a^d=((xa-1)aa^d)a^d=0.$
By virtue of Theorem 1.2, $a^{\tiny\textcircled{d}}=aa^dz$ for some $z\in \mathcal{A}$, and then $im(x)\subseteq im(a^d)$, as desired.

$(2)\Rightarrow (3)$ By hypothesis, we have $im(x)\subseteq im(a^d)=im(xaa^d)\subseteq im(x)$, and so $im(x)=im(a^d)$.
Clearly, $x^*=(xax)^*=(ax)^*x^*=(ax)x^*\in im(aa^d)=im(a^d)$, and then $im(x^*)\subseteq im(a^d)$. Moreover,
$a^d=(aa^d)a^d=a(xa^2a^d)a^d=(ax)(aa^d)=(ax)^*(aa^d)=x^*(a^*aa^d)$; hence, $im(a^d)\subseteq im(x^*)$.
This implies that $im(x^*)=im(a^d)$, as desired.

$(3)\Rightarrow (4)$ This is obvious.

$(4)\Rightarrow (1)$ Since $xax=x$, we have $(xa-1)x=0$, and then $xa-1\in \ell(x)\subseteq ell(a^d)$.
Hence $(xa-1)a^d=0$, and so $xaa^d=a^d$.
Observe that $$\begin{array}{rll}
a^n-xa^{n+1}&=&a^n-a^da^{n+1}+a^da^{n+1}-xa^{n+1}\\
&=&a^n-a^da^{n+1}+xa^da^{n+2}-xa^{n+1}\\
&=&a^n-a^da^{n+1}-x(a^n-a^da^{n+1})a.
\end{array}$$ Then we have $$||a^n-xa^{n+1}||^{\frac{1}{n}}\leq (1+||x||)^{\frac{1}{n}}||a||^{\frac{1}{n}}||a^n-a^da^{n+1}||^{\frac{1}{n}}.$$
Since $\lim_{n\to \infty}||a^n-a^da^{n+1}||^{\frac{1}{n}}=0$, we see that $\lim_{n\to \infty}||a^n-xa^{n+1}||^{\frac{1}{n}}=0.$
Moreover, we have $x^*a^*x^*=x^*$; whence, $(1-x^*a^*)x^*=0$. Thus $1-x^*a^*\in \ell(x^*)\subseteq \ell(a^d)$, and so
$(1-x^*a^*)a^d=0$. We infer that
$x^*a^*a^d=a^d$, and then $(ax)^*a^d=a^d$. Therefore $(ax)^*aa^d=aa^d$. Since $\ell(a^d)\subseteq \ell(x)$, we have $(ax)^*ax=ax$; and so
$(ax)^*=(ax)^*ax=ax$. As $(ax)a^d=(ax)^*a^d=x^*a^*a^d=a^d$, we have $ax-1\in \ell(a^d)\subseteq \ell(x)$, we have $ax^2=x$.
This completes the proof.\end{proof}

The preceding theorem infers that the generalized core-EP inverse and core-EP inverse for a bounded linear operator over a Hilbert space and an element in a ring introduced in~\cite{MD4,MD3} coincide with each other.

\begin{cor} Let $a\in \mathcal{A}$. Then the following are equivalent:\end{cor}
\begin{enumerate}
\item [(1)] $a\in \mathcal{A}^{\tiny\textcircled{d}}$.
\vspace{-.5mm}
\item [(2)] $a\in \mathcal{A}^d$ and $\mathcal{A}=a^d\mathcal{A}\oplus \ell{(a^d)}=(a^d)^*\mathcal{A}\oplus \ell{(a^d)}.$
\vspace{-.5mm}
\item [(3)] $a\in \mathcal{A}^d$ and $\mathcal{A}=\mathcal{A}a^d\oplus r(a^d)=\mathcal{A}(a^d)^*\oplus r(a^d).$
\end{enumerate}
\begin{proof} $(1)\Leftrightarrow (2)$ Since $\mathcal{A}=a^d\mathcal{A}\oplus \ell{(a^d)}=(a^d)^*\mathcal{A}\oplus \ell{(a^d)},$ it follows by
~\cite[Proposition 4.2]{Z1} that $a^d\mathcal{A}=q\mathcal{A}$ for some projection $q\in \mathcal{A}$. The equivalence is proved by Theorem 2.1.

$(1)\Leftrightarrow (3)$ This is analogously prove by the symmetry.\end{proof}

\begin{thm} Let $a\in \mathcal{A}$. Then $a\in \mathcal{A}^{\tiny\textcircled{d}}$ if and only if\end{thm}
\begin{enumerate}
\item [(1)] $a\in \mathcal{A}^d$;
\vspace{-.5mm}
\item [(2)] There exists a projection $q\in \mathcal{A}$ such that $\ell{(a^d)}=\ell{(q)}.$
\end{enumerate}
In this case, $a^{\tiny\textcircled{d}}=a^dq.$
\begin{proof}  $\Longrightarrow $ In view of Theorem 1.2, $a^d\in \mathcal{A}^{\tiny\textcircled{\#}}$ and
$a^{\tiny\textcircled{d}}=(a^d)^2(a^d)^{\tiny\textcircled{\#}}.$ Let $q=a^d(a^d)^{\tiny\textcircled{\#}}$. Then $q^2=q=q^*$, i.e., $p\in \mathcal{A}$ is a
projection. If $xa^d=0$, then $xq=(xa^d)(a^d)^{\tiny\textcircled{\#}}=0$. If $ya^d(a^d)^{\tiny\textcircled{\#}}=0$, then
$ya^d=[ya^d(a^d)^{\tiny\textcircled{\#}}]a^d=0$. Therefore $\ell{(a^d)}=\ell{(q)}.$ Moreover, we have $$
a^{\tiny\textcircled{d}}=(a^d)^2(a^d)^{\tiny\textcircled{\#}}=a^d(a^d(a^d)^{\tiny\textcircled{\#}})=a^dq,$$ as desired.

$\Longleftarrow $ By hypothesis, there exists a projection $q\in \mathcal{A}$ such that $\ell{(a^d)}=\ell{(q)}.$ Since $1-aa^d\in \ell{(a^d)}$, we have
$(1-aa^d)q=0$, and then $q=aa^dq$. As $1-q\in \ell(q)$, we get $(1-q)a^d=0$; hence, $a^d=qa^d$. It is easy to verify that
$$\begin{array}{c}
a^d(aq)=q, [a^d(aq)]^*=a^d(aq), (aq)(a^d)^2=a(qa^d)a^d=a(a^d)^2=a^d,\\
a^d(aq)^2=(a^daq)aq=aq-(1-q)aq=aq-(1-q)a^da^2=aq.
\end{array}$$ This implies that $a^d\in \mathcal{A}^{\tiny\textcircled{\#}}$. According to Theorem 1.2,
$a\in \mathcal{A}^{\tiny\textcircled{d}}$.\end{proof}

\begin{cor} Let $a\in \mathcal{A}$. Then $a\in \mathcal{A}^{\tiny\textcircled{d}}$ if and only if\end{cor}
\begin{enumerate}
\item [(1)] $a\in \mathcal{A}^d$;
\vspace{-.5mm}
\item [(2)] There exists a unique projection $q\in \mathcal{A}$ such that $\ell{(a^d)}=\ell{(q)}$.
\end{enumerate}
\begin{proof} $\Longrightarrow $ In view of Theorem 2.3, we can find a projection $q\in \mathcal{A}$ such that $\ell{(a^d)}=\ell{(q)}$. If
$\ell{(a^d)}=\ell{(p)}$ for a projection $p\in \mathcal{A}$. As in the proof of Theorem 2.3, we have
$$q=aa^dq, p=aa^dp, a^d=a^dq=a^dp.$$ Therefore $p=a(a^dp)=a(a^dq)=q,$ as required.

$\Longleftarrow $ This is obvious by Theorem 2.3.\end{proof}

\section{reverse order law and related properties}

In this section we establish the reverse order for generalized core-EP inverse in a Banach *-algebra. Let $a,b,c\in \mathcal{A}$. An element
$a$ has $(b,c)$-inverse provide that there exists $x\in \mathcal{A}$ such that $$xab=b, cax=c ~\mbox{and}~ x\in b\mathcal{A}x\bigcap x\mathcal{A}c.$$ If
such $x$ exists, it is unique and denote it by $a^{(b,c)}$ (see~\cite{D1}).

\begin{lem} Let $a\in \mathcal{A}$. Then the following are equivalent:\end{lem}
\begin{enumerate}
\item [(1)] $a\in \mathcal{A}^{\tiny\textcircled{d}}$.
\vspace{-.5mm}
\item [(2)] $a\in \mathcal{A}^d$ and $a$ has $(aa^d, (aa^d)^*)$-inverse.
\end{enumerate}
In this case, $a^{\tiny\textcircled{d}}=a^{(aa^d, (aa^d)^*)}.$
\begin{proof} See~\cite[Theorem 2.3]{MD4}.\end{proof}

Let ${\Bbb C}^*$ be the set of all nonzero complex numbers. For future use, We now record the following.

\begin{lem} Let $a,x\in \mathcal{A}$, $\lambda,\mu\in {\Bbb C}^*, ax=\lambda xa$ and $a^*x=\mu xa^*$. If $a\in \mathcal{A}^{\tiny\textcircled{d}}$, then
$a^{\tiny\textcircled{d}}x=\lambda^{-1}xa^{\tiny\textcircled{d}}.$\end{lem}
\begin{proof} In view of Lemma 3.1, $a^{\tiny\textcircled{d}}$ is $(aa^d, (aa^d)^*)$-inverse of $a$.
That is, $a^{\tiny\textcircled{d}}=a^{(aa^d, (aa^d)^*)}.$ One directly checks that $$\lambda^{-1}a^{\tiny\textcircled{d}}=(\lambda a)^{(\lambda^{-1}aa^d,\mu^{-1}(aa^d)^*)}.$$
Since $ax=\lambda xa$, it follows by ~\cite[Theorem 15.2.12]{C1} that $a^dx=\lambda^{-1}xa^d=x(\lambda^{-1}a^d).$ On the other hand, $a^*x=\mu xa^*$, and so
$$(a^d)^*x=(a^*)^dx=\mu^{-1}x(a^*)^d=x(\mu^{-1}(a^d)^*).$$ In light of Lemma 3.1 and ~\cite[Theorem 2.3]{D}, we have $$a^{\tiny\textcircled{d}}x=x(\lambda^{-1} a^{\tiny\textcircled{d}})=\lambda^{-1} xa^{\tiny\textcircled{d}},$$ as required.\end{proof}

We are now ready to prove:

\begin{thm} Let $a,b\in \mathcal{A}^{\tiny\textcircled{d}}$, $\lambda,\mu\in {\Bbb C}^*$. If $ab=\lambda ba, a^*b=\mu ba^*$, then $ab\in \mathcal{A}^{\tiny\textcircled{d}}$. In this case,
$$(ab)^{\tiny\textcircled{d}}=b^{\tiny\textcircled{d}}a^{\tiny\textcircled{d}}=\lambda^{-1}a^{\tiny\textcircled{d}}b^{\tiny\textcircled{d}}.$$\end{thm}
\begin{proof} In view of lemma 3.2, we have $ab^{\tiny\textcircled{d}}=\lambda^{-1}b^{\tiny\textcircled{d}}a$ and
$a^*b^{\tiny\textcircled{d}}=\mu^{-1}b^{\tiny\textcircled{d}}a^*$. Then $a^{\tiny\textcircled{d}}b^{\tiny\textcircled{d}}=\lambda b^{\tiny\textcircled{d}}a^{\tiny\textcircled{d}}.$ Similarly, we have $ba^{\tiny\textcircled{d}}=\lambda a^{\tiny\textcircled{d}}b$.

Case 1. $||\lambda||\geq 1$. Set $x=b^{\tiny\textcircled{d}}a^{\tiny\textcircled{d}}.$ Then we check that
$$\begin{array}{rll}
abx&=&abb^{\tiny\textcircled{d}}a^{\tiny\textcircled{d}}=\lambda^{-1}aba^{\tiny\textcircled{d}}b^{\tiny\textcircled{d}}
=aa^{\tiny\textcircled{d}}bb^{\tiny\textcircled{d}},\\
(abx)^*&=&bb^{\tiny\textcircled{d}}aa^{\tiny\textcircled{d}}=\lambda bab^{\tiny\textcircled{d}}a^{\tiny\textcircled{d}}=baa^{\tiny\textcircled{d}}b^{\tiny\textcircled{d}}\\
&=&\lambda^{-1}aba^{\tiny\textcircled{d}}b^{\tiny\textcircled{d}}=aa^{\tiny\textcircled{d}}bb^{\tiny\textcircled{d}}=abx,\\
abx^2&=&(\lambda^{-1}aba^{\tiny\textcircled{d}}b^{\tiny\textcircled{d}})b^{\tiny\textcircled{d}}a^{\tiny\textcircled{d}}=aa^{\tiny\textcircled{d}}b(b^{\tiny\textcircled{d}})^2a^{\tiny\textcircled{d}}\\
&=&aa^{\tiny\textcircled{d}}b^{\tiny\textcircled{d}}a^{\tiny\textcircled{d}}=\lambda ab^{\tiny\textcircled{d}}a^{\tiny\textcircled{d}}a^{\tiny\textcircled{d}}\\
&=&b^{\tiny\textcircled{d}}a(a^{\tiny\textcircled{d}})^2=b^{\tiny\textcircled{d}}a^{\tiny\textcircled{d}}=x,\\
(ab)^{m}&=&\lambda^{-\frac{(m-1)m)}{2}}a^{m}b^{m}, (ab)^{m+1}=\lambda^{-\frac{m(m+1)}{2}}a^{m+1}b^{m+1},\\
b^{\tiny\textcircled{d}}a^{m+1}&=&\lambda^{m+1}ab^{\tiny\textcircled{d}},\\
(ab)^m-x(ab)^{m+1}&=&\lambda^{-\frac{(m-1)m}{2}}a^mb^m-\lambda^{-\frac{m(m+1)}{2}}b^{\tiny\textcircled{d}}a^{\tiny\textcircled{d}}a^{m+1}b^{m+1}\\
&=&\lambda^{-\frac{(m-1)m}{2}}[a^mb^m-\lambda^{-m}b^{\tiny\textcircled{d}}a^{\tiny\textcircled{d}}a^{m+1}b^{m+1}]\\
&=&\lambda^{-\frac{(m-1)m}{2}}[a^mb^m-\lambda^{-(m+1)}a^{\tiny\textcircled{d}}b^{\tiny\textcircled{d}}a^{m+1}b^{m+1}]\\
&=&\lambda^{-\frac{(m-1)m}{2}}[a^mb^m-a^{\tiny\textcircled{d}}a^{m+1}b^{\tiny\textcircled{d}}b^{m+1}]\\
&=&\lambda^{-\frac{(m-1)m}{2}}[(a^m-a^{\tiny\textcircled{d}}a^{m+1})b^m+a^{\tiny\textcircled{d}}a^{m+1}(b^m-b^{\tiny\textcircled{d}}b^{m+1})].
\end{array}$$ Hence,
$$\begin{array}{rl}
&||(ab)^m-x(ab)^{m+1}||^{\frac{1}{m}}\\
\leq& ||a^m-a^{\tiny\textcircled{d}}a^{m+1}||^{\frac{1}{m}}||b^m||^{\frac{1}{m}}+
||a^{\tiny\textcircled{d}}a^{m+1}||^{\frac{1}{m}}||b^m-b^{\tiny\textcircled{d}}b^{m+1}||^{\frac{1}{m}}.
\end{array}$$
Then $$\lim_{n\to \infty}||(ab)^m-x(ab)^{m+1}||^{\frac{1}{m}}=0.$$ Therefore $ab\in \mathcal{A}^{\tiny\textcircled{d}}$ and
$(ab)^{\tiny\textcircled{d}}=x$.

Case 2. $||\lambda||\leq 1$. Then $ba=\lambda^{-1}ab, ba^*=\mu^{-1}a^*b$. Since $||\lambda^{-1}||\geq 1$, as in the proof of Step 1, we have
$ba\in \mathcal{A}^{\tiny\textcircled{d}}$ and $(ba)^{\tiny\textcircled{d}}=a^{\tiny\textcircled{d}}b^{\tiny\textcircled{d}}$. In view of Theorem 2.1,
$ba\in \mathcal{A}^d$ and $ba(ba)^d\in \mathcal{A}^{(1,3)}$. By virtue of Cline's formula (see~\cite[Theorem 2.1]{L}),
$ab\in \mathcal{A}^d$. Clearly, $a(ba)=\lambda (ba)a$, and then $a(ba)^d=\lambda^{-1}(ba)^da$.
By using Cline's formula again, we have $$ab(ab)^d=aba[(ba)^d]^2b=a(ba)^db=\lambda^{-1}(ba)^d(ab)=(ba)^d(ba)=ba(ba)^d.$$
Hence $(ab)(ab)^d\in \mathcal{A}^{(1,3)}$. In light of Theorem 2.1, $ab\in \mathcal{A}^{\tiny\textcircled{d}}$.
Further, we check that $$\begin{array}{rll}
(ab)^{\tiny\textcircled{d}}&=&(ab)^d[(ab)(ab)^d]^{(1,3)}=(ab)^d[(ba)(ba)^d]^{(1,3)}=(ab)^d[(ba)(ba)^{\tiny\textcircled{d}}]\\
&=&(ab)^d(ba)a^{\tiny\textcircled{d}}b^{\tiny\textcircled{d}}=b^da^d(ba)a^{\tiny\textcircled{d}}b^{\tiny\textcircled{d}}=\lambda^{-1}b^dbaa^da^{\tiny\textcircled{d}}b^{\tiny\textcircled{d}}\\
&=&\lambda^{-1}b^dba^{\tiny\textcircled{d}}b^{\tiny\textcircled{d}}=b^dbb^{\tiny\textcircled{d}}a^{\tiny\textcircled{d}}=b^{\tiny\textcircled{d}}a^{\tiny\textcircled{d}},\\
\end{array}$$ thus yielding the result.\end{proof}

\begin{lem} Let $a\in \mathcal{A}$ and $k\in {\Bbb N}$. Then the following are equivalent:\end{lem}
\begin{enumerate}
\item [(1)] $a\in \mathcal{A}^{\tiny\textcircled{d}}$.
\vspace{-.5mm}
\item [(2)] $a^k\in \mathcal{A}^{\tiny\textcircled{d}}$.
\end{enumerate}
In this case, $$(a^k)^{\tiny\textcircled{d}}=(a^{\tiny\textcircled{d}})^k, a^{\tiny\textcircled{d}}=a^{k-1}(a^k)^{\tiny\textcircled{d}}.$$
\begin{proof} $(1)\Rightarrow (2)$ See~\cite[Lemma 2.4]{MD4}.

$(2)\Rightarrow (1)$ $(2)\Rightarrow (1)$ Let $y=(a^k)^{\tiny\textcircled{d}}$. Then $$y=a^ky^2, (a^ky)^*=a^ky, \lim_{n\to \infty}||(a^k)^n-y(a^k)^{n+1}||^{\frac{1}{n}}=0.$$
Let $x=a^{k-1}y$. We verify that $$\begin{array}{rll}
ax^2&=&a^kya^{k-1}y=a^kya^{k-1}(a^ky^2)\\
&=&a^kya^{k-1}(a^k)^{m+1}y^{m+2}=a^k[y(a^k)^{m+1}]a^{k-1}y^{m+2}\\
&=&a^k[y(a^k)^{m+1}-(a^k)^m]a^{k-1}y^{m+2}\\
&+&(a^k)^{m+1}a^{k-1}y^{m+2}=(a^k)^{m+1}a^{k-1}y^{m+2}\\
&=&a^{k-1}y=x,\\
(ax)^*&=&(a^ky)^*=a^ky=ax,\\
||a^n-xa^{n+1}||^{\frac{1}{n}}&=&||a^n-a^{k-1}ya^{n+1}||^{\frac{1}{n}}\\
&\leq &||a^{k-1}||^{\frac{1}{n}}||a^{n-k+1}-ya^{n+1}||^{\frac{1}{n}}\\
\end{array}$$
Then $\lim_{n\to \infty}||a^n-xa^{n+1}||^{\frac{1}{n}}=0$. Accordingly, $a\in \mathcal{A}^{\tiny\textcircled{d}}$ and
$a^{\tiny\textcircled{d}}=a^{k-1}(a^k)^{\tiny\textcircled{d}}$.\end{proof}

We are now ready to prove:

\begin{thm} Let $a,b\in \mathcal{A}^{\tiny\textcircled{d}}$. If $bab=\lambda ab^2=\mu b^2a$ and $ba^*b=\lambda'a^*b^2=\mu'b^2a^*$, then $ab\in
\mathcal{A}^{\tiny\textcircled{d}}$. In this case,
$$(ab)^{\tiny\textcircled{d}}=\mu^{-1}b^{\tiny\textcircled{d}}a^{\tiny\textcircled{d}}.$$\end{thm}
\begin{proof} In view of Lemma 3.4, $a^2,b^2\in \mathcal{A}^{\tiny\textcircled{d}}$. By hypothesis, we check that
$$\begin{array}{rll}
a^2b^2&=&a(ab^2)=\lambda^{-1}a(bab)=\lambda^{-1}\mu a(b^2a)=\lambda^{-1}\mu (ab^2)a\\
&=&\lambda^{-2}\mu^2(b^2a)a=\lambda^{-2}\mu^2 b^2a^2.
\end{array}$$ Likewise, $(a^2)^*b^2=(\lambda')^{-2}(\mu')^2b^2(a^2)^*$. By virtue of Theorem 3.3, $a^2b^2\in
\mathcal{A}^{\tiny\textcircled{d}}$ and
$$(a^2b^2)^{\tiny\textcircled{d}}=(b^2)^{\tiny\textcircled{d}}(a^2)^{\tiny\textcircled{d}}=\lambda^{2}\mu^{-2}(a^2)^{\tiny\textcircled{d}}(b^2)^{\tiny\textcircled{d}}.$$
Obviously, $(ab)^2=a(bab)=\lambda a^2b^2$, and so $(ab)^2\in \mathcal{A}^{\tiny\textcircled{d}}$. In view of Lemma 3.4,
we have $ab\in \mathcal{A}^{\tiny\textcircled{d}}$. Moreover, we have
$$\begin{array}{rll}
[(ab)^2]^{\tiny\textcircled{d}}&=&[\lambda (a^2b^2)]^{\tiny\textcircled{d}}=[\lambda (a^2b^2)]^d[\lambda (a^2b^2)\lambda^{-1}(a^2b^2)]^d]^{(1,3)}\\
&=&[\lambda (a^2b^2)]^d[(a^2b^2)(a^2b^2)]^d]^{(1,3)}=\lambda^{-1}(a^2b^2)^{\tiny\textcircled{d}}\\
&=&\lambda^{-1}(b^2)^{\tiny\textcircled{d}}(a^2)^{\tiny\textcircled{d}}=\lambda\mu^{-2}(a^2)^{\tiny\textcircled{d}}(b^2)^{\tiny\textcircled{d}}.
\end{array}$$
We easily checks that $$ab^2=\lambda^{-1}\mu b^2a, a^*b^2=(\lambda')^{-1}(\mu')b^2a.$$
In light of Lemma 3.2, $a(b^2)^{\tiny\textcircled{d}}=\lambda\mu^{-1}(b^2a)^{\tiny\textcircled{d}}.$
By using Lemma 3.4 again, we have
$$\begin{array}{rll}
(ab)^{\tiny\textcircled{d}}&=&ab[(ab)^{\tiny\textcircled{d}}]^2\\
&=&ab[(ab)^2]^{\tiny\textcircled{d}}\\
&=&\lambda^{-1}ba(b^2)^{\tiny\textcircled{d}}(a^2)^{\tiny\textcircled{d}}\\
&=&\mu^{-1}b(b^2)^{\tiny\textcircled{d}}a(a^2)^{\tiny\textcircled{d}}\\
&=&\mu^{-1}[b(b^2)^{\tiny\textcircled{d}}][a(a^2)^{\tiny\textcircled{d}}]\\
&=&\mu^{-1}b^{\tiny\textcircled{d}}a^{\tiny\textcircled{d}}.
\end{array}$$ This completes the proof.\end{proof}

Let $\mathcal{A}$ be a Banach *-algebra. Then $M_2(\mathcal{A})$ is a Banach *-algebra with *-transpose as the involution. We come now to consider the generalized EP-inverse of a triangular matrix over $\mathcal{A}$.

\begin{thm} Let $x=\left(
\begin{array}{cc}
a&b\\
0&d
\end{array}
\right)\in M_2(\mathcal{A})$ with $a,d\in \mathcal{A}^{\tiny\textcircled{d}}$. If $$\sum\limits_{i=0}^{\infty}a^ia^{\pi}b(d^d)^{i+2}=0,$$ then $x\in
M_2(\mathcal{A})^{\tiny\textcircled{d}}$ and $$x^{\tiny\textcircled{d}}=\left(
\begin{array}{cc}
a^{\tiny\textcircled{d}}&z\\
0&d^{\tiny\textcircled{d}}
\end{array}
\right),$$ where $z=a^d(a^d)^{\tiny\textcircled{d}}bdd^{\tiny\textcircled{d}}-(a^d)^2b(d^d)^3-a^db(d^d)^4.$\end{thm}
\begin{proof} In view of Theorem 1.2, $a,d\in \mathcal{A}^d$ and $a^d,d^d\in \mathcal{A}^{\tiny\textcircled{\#}}$. By virtue of~\cite[Lemma 15.2.1]{C1}, we have $$x^{d}=\left(
\begin{array}{cc}
a^d&s\\
0&d^d
\end{array}
\right),$$ where $s=\sum\limits_{i=0}^{\infty}(a^d)^{i+2}bd^id^{\pi}+\sum\limits_{i=0}^{\infty}a^ia^{\pi}b(d^d)^{i+2}-a^dbd^d.$ By hypothesis, we get
$s=\sum\limits_{i=0}^{\infty}(a^d)^{i+2}bd^id^{\pi}-a^dbd^d.$ Since $(a^d)^{\pi}s=(1-a^da^2a^d)s=a^{\pi}s=\sum\limits_{i=0}^{\infty}a^ia^{\pi}b(d^d)^{i+2}=0$, we have $[1-(a^d)^{\tiny\textcircled{\#}}(a^d)]s=[1-(a^d)^{\#}a^d(a^d)^{(1,3)}a^d]s=(a^d)^{\pi}s=0$. In view of~\cite[Lemma 2.4]{XS},
we have $[1-a^d(a^d)^{\tiny\textcircled{\#}}]s=0$. Then it follows by~\cite[Theorem 2.5]{XS} that
$$(x^d)^{\tiny\textcircled{\#}}=\left(
\begin{array}{cc}
(a^d)^{\tiny\textcircled{\#}}&t\\
0&(d^d)^{\tiny\textcircled{\#}}
\end{array}
\right),$$ where $t=-(a^d)^{\tiny\textcircled{\#}}s(d^d)^{\tiny\textcircled{\#}}.$ Hence,
$t=-(a^d)^{\tiny\textcircled{\#}}[\sum\limits_{i=0}^{\infty}(a^d)^{i+2}bd^id^{\pi}-a^dbd^d](d^d)^{\tiny\textcircled{\#}}
=(a^d)^{\tiny\textcircled{\#}}a^db(d^d)(d^d)^{\tiny\textcircled{\#}}=a^d(a^d)^{\tiny\textcircled{\#}}b(d^d)(d^d)^{\tiny\textcircled{\#}}.$
Then we have $$(x^d)^2=\left(
\begin{array}{cc}
(a^d)^2&w\\
0&(d^d)^2
\end{array}
\right),$$ where $w=\sum\limits_{i=0}^{\infty}(a^d)^{i+3}bd^id^{\pi}-(a^d)^2bd^d-a^db(d^d)^2.$
Therefore $$\begin{array}{rll}
x^{\tiny\textcircled{d}}&=&(x^d)^2(x^d)^{\tiny\textcircled{\#}}\\
&=&\left(
\begin{array}{cc}
(a^d)^2&w\\
0&(d^d)^2
\end{array}
\right)\left(
\begin{array}{cc}
(a^d)^{\tiny\textcircled{\#}}&t\\
0&(d^d)^{\tiny\textcircled{\#}}
\end{array}
\right)\\
&=&\left(
\begin{array}{cc}
a^{\tiny\textcircled{d}}&z\\
0&d^{\tiny\textcircled{d}}
\end{array}
\right),
\end{array}$$ where $$\begin{array}{rll}
z&=&(a^d)^2t+w(d^d)^2\\
&=&(a^d)^2[a^d(a^d)^{\tiny\textcircled{\#}}b(d^d)(d^d)^{\tiny\textcircled{\#}}]-[(a^d)^2bd^d+a^db(d^d)^2](d^d)^2\\
&=&(a^d)^3(a^d)^{\tiny\textcircled{\#}}b(d^d)(d^d)^{\tiny\textcircled{\#}}-(a^d)^2b(d^d)^3-a^db(d^d)^4\\
&=&a^d(a^d)^{\tiny\textcircled{d}}bdd^{\tiny\textcircled{d}}-(a^d)^2b(d^d)^3-a^db(d^d)^4.
\end{array}$$ This completes the proof.\end{proof}

It is very hard to determine the core-EP inverse of a triangular complex matrix (see~\cite{MD2}). As a consequence of Theorem 3.6, we now derive the
following.

\begin{cor} Let $M=\left(
\begin{array}{cc}
A&B\\
0&D
\end{array}
\right)$, $A,B,D\in {\Bbb C}^{n\times n}$. If $$\sum\limits_{i=0}^{i(A)}A^iA^{\pi}B(D^D)^{i+2}=0,$$ then $$M^{\tiny\textcircled{D}}=\left(
\begin{array}{cc}
A^{\tiny\textcircled{D}}&Z\\
0&D^{\tiny\textcircled{D}}
\end{array}
\right),$$ where $Z=A^DA^{\tiny\textcircled{D}}BDD^{\tiny\textcircled{D}}-(A^D)^2B(D^D)^3-A^DB(D^D)^4.$\end{cor}
\begin{proof} Since the generalized core-EP inverse and core-EP inverse coincide with each other for a complex matrix, we obtain the result by Theorem
3.6.\end{proof}

\section{Generalized core-EP orders}

This section is devoted to the generalized core-EP order for Banach *-algebra elements.  Let $e=\{ e_1,\cdots ,e_n\}$ be the set of idempotents $e_i$ satisfying $1=e_1+\cdots +e_n$ and $e_ie_j=0 (i\neq j)$. Then for any $x\in \mathcal{A}$, we have $x=\sum\limits_{i,j=1}^{n}e_ixe_j$. Thus $x$ can be represented in the matrix form
$x=(e_ixf_j)_{e\times e}$. When $n=2$ and $p,q=1-p$ are orthogonal idempotents, we may write $x=\left(
\begin{array}{cc}
pxp&px(1-p)\\
(1-p)xp&(1-p)x(1-q)
\end{array}
\right)_p$. An element $a\in \mathcal{A}$ has generalized core-EP decomposition if there exist $x,y\in \mathcal{A}$ such that $$a=x+y, x^*y=yx=0, x\in
\mathcal{A}^{\tiny\textcircled{\#}}, y\in \mathcal{A}^{qnil}.$$ In view of Theorem 1.2,  $a\in \mathcal{A}^{\tiny\textcircled{d}}$ if and only if it has generalized core-EP decomposition. We are now ready to prove:

\begin{lem} Let $a=a_1+a_2$ be the generalized core-EP decomposition of $a$, and let $p=a_1a_1^{\tiny\textcircled{\#}}$. Then
$$a=\left(
\begin{array}{cc}
t&s\\
0&n
\end{array}
\right)_p,$$ where $t\in (p\mathcal{A}p)^{-1}, n\in \big(p^{\pi}\mathcal{A}p^{\pi})^{qnil}.$\end{lem}
\begin{proof} Obviously, $p^2=p=p^*$. It is easy to verify that $p^{\pi}a_1=(1-a_1a_1^{\tiny\textcircled{\#}})a_1=0$, and so
$a_1=\left(
\begin{array}{cc}
t&s\\
0&0
\end{array}
\right)_p.$ Moreover, we have $pa_2=a_1a_1^{\tiny\textcircled{\#}}a_2=(a_1^{\tiny\textcircled{\#}})^*(a_1)^*a_2=0$ and
$p^{\pi}a_2p=(1-a_1a_1^{\tiny\textcircled{\#}})a_2a_1a_1^{\tiny\textcircled{\#}}=0$. Hence, $a_2=\left(
\begin{array}{cc}
0&0\\
0&n
\end{array}
\right)_p,$ where $n=p^{\pi}a_2p^{\pi}\in \mathcal{A}^{qnil}$. Therefore $a=a_1+a_2=\left(
\begin{array}{cc}
t&s\\
0&n
\end{array}
\right)_p,$ as asserted.\end{proof}

\begin{lem} Let $a=a_1+a_2$ be the generalized core-EP decomposition of $a$, and let $p=aa^{\tiny\textcircled{d}}$. Then
$$a=\left(
\begin{array}{cc}
t_1&t_2\\
0&a_2
\end{array}
\right)_p,$$ where $t_1\in (p\mathcal{A}p)^{-1}, a_2\in [(1-p)\mathcal{A}(1-p)]^{qnil}.$ For any $b\in \mathcal{A}^{\tiny\textcircled{d}}$, we have
$$a\leq^{{\tiny\textcircled{d}}}b~\mbox{if and only if} ~b=\left(
\begin{array}{cc}
t_1&t_2\\
0&c
\end{array}
\right)_p,$$ where $c\in (1-p)\mathcal{A}(1-p).$\end{lem}
\begin{proof} By using Lemma 4.1, we have $a=\left(
\begin{array}{cc}
t_1&t_2\\
0&a_2
\end{array}
\right)_p,$ where $t_1\in (p\mathcal{A}p)^{-1},$ $a_2\in [(1-p)\mathcal{A}(1-p)]^{qnil}.$ Let $b\in \mathcal{A}^{\tiny\textcircled{d}}$. If
$a\leq^{{\tiny\textcircled{d}}}b$, then
$aa^{{\tiny\textcircled{d}}}=ba^{{\tiny\textcircled{d}}}$ and
$a^{{\tiny\textcircled{d}}}a=a^{{\tiny\textcircled{d}}}b$. We verify that
$$(1-p)bp=(1-aa^{\tiny\textcircled{d}})ba^{\tiny\textcircled{d}}a=(1-aa^{\tiny\textcircled{d}})aa^{\tiny\textcircled{d}}a=0.$$
Then $b=\left(
\begin{array}{cc}
t_1&t_2\\
0&c
\end{array}
\right)_p,$ where $c\in (1-p)\mathcal{A}(1-p).$

Conversely, assume that $a$ and $b$ has the matrix form above. Then $$a=\left(
\begin{array}{cc}
t_1&t_2\\
0&a_2
\end{array}
\right)_p, a^{\tiny\textcircled{d}}=\left(
\begin{array}{cc}
a^{\tiny\textcircled{d}}&0\\
0&0
\end{array}
\right)_p,$$ we directly verify that $aa^{\tiny\textcircled{d}}=ba^{\tiny\textcircled{d}}$ and
$a^{\tiny\textcircled{d}}a=a^{\tiny\textcircled{d}}b$. Accordingly, $a\leq^{\tiny\textcircled{d}}b$.\end{proof}

\begin{lem} Let $a,b\in \mathcal{A}^{\tiny\textcircled{d}}$ and $p=aa^{\tiny\textcircled{d}}$. Then
$(1-aa^{\tiny\textcircled{d}})b(1-aa^{\tiny\textcircled{d}})\in \mathcal{A}^{\tiny\textcircled{d}}$.\end{lem}
\begin{proof} Step 1.  Write $b^{\tiny\textcircled{d}}=\left(
\begin{array}{cc}
x_1&x_2\\
x_3&x_4
\end{array}
\right)_p.$ Construct $t_1,t_2$ as in Lemma 4.2, it follows by Lemma 4.2 that $b=\left(
\begin{array}{cc}
t_1&t_2\\
0&c
\end{array}
\right)_p$, where $c\in (1-p)\mathcal{A}(1-p).$ Then $$bb^{\tiny\textcircled{d}}=\left(
\begin{array}{cc}
t_1&t_2\\
0&c
\end{array}
\right)_p\left(
\begin{array}{cc}
x_1&x_2\\
x_3&x_4
\end{array}
\right)_p=\left(
\begin{array}{cc}
*&*\\
*&cx_4
\end{array}
\right)_p.$$ Likewise, we have $$(b^{\tiny\textcircled{d}})^*b^*=\left(
\begin{array}{cc}
*&*\\
*&(x_4)^*c^*
\end{array}
\right)_p.$$ Since $(bb^{\tiny\textcircled{d}})^*=bb^{\tiny\textcircled{d}}$, we have
$cx_4=(x_4)^*c^*=(cx_4)^*.$ Clearly, $cx_4=(1-p)b(1-p)b^{\tiny\textcircled{d}}(1-p)$. Hence,
$[(1-p)b(1-p)b^{\tiny\textcircled{d}}(1-p)]^*=(1-p)b(1-p)b^{\tiny\textcircled{d}}(1-p).$

Step 2. Obviously, $aa^{\tiny\textcircled{d}}=aa^{\tiny\textcircled{d}}bb^{\tiny\textcircled{d}}=bb^{\tiny\textcircled{d}}aa^{\tiny\textcircled{d}}.$
Then $(1-p)bb^{\tiny\textcircled{d}}=bb^{\tiny\textcircled{d}}(1-p)$.
Thus $$\begin{array}{rll}
&(1-p)b(1-p)[(1-p)b^{\tiny\textcircled{d}}(1-p)]^2\\
=&[(1-p)b(1-p)b^{\tiny\textcircled{d}}(1-p)][b^{\tiny\textcircled{d}}(1-p)]\\
=&(cx_4)[b^{\tiny\textcircled{d}}(1-p)]\\
=&[(1-p)bb^{\tiny\textcircled{d}}(1-p)][b^{\tiny\textcircled{d}}(1-p)]\\
=&(1-p)b(b^{\tiny\textcircled{d}})^2(1-p)\\
=&(1-p)b^{\tiny\textcircled{d}}(1-p).
\end{array}$$

Step 3. For any $m\in {\Bbb N}$, we have
$b^m=\left(
\begin{array}{cc}
t_1^m&*\\
0&c^m
\end{array}
\right)_p,$ and then $$b^m-b^{\tiny\textcircled{d}}b^{m+1}=\left(
\begin{array}{cc}
*&*\\
x_3t_1^{m+1}&*
\end{array}
\right)_p.$$ Since $\lim_{m\to \infty}||b^m-b^{\tiny\textcircled{d}}b^{m+1}||^{\frac{1}{m}}=0$, we deduce that
$\lim_{m\to \infty}||x_3t_1^{m+1}||^{\frac{1}{m}}=0$. Obviously, $a^{\tiny\textcircled{d}}t_1=t_1a^{\tiny\textcircled{d}}=aa^{\tiny\textcircled{d}}$.
Then $\lim_{m\to \infty}||x_3||^{\frac{1}{m}}=0$. This implies that $x_3=0$, and then  $(1-p)b^{\tiny\textcircled{d}}p=0$.
For any $m\in {\Bbb N}$, we have $b^m=\left(
\begin{array}{cc}
t_1^m&*\\
0&c^m
\end{array}
\right)_p$; hence, $(1-p)b^mp=0$. We infer that $[(1-p)b(1-p)]^{m}=(1-p)b^m$ and $[(1-p)b(1-p)]^{m+1}=(1-p)b^{m+1}$.
Thus, we have $$[(1-p)b^{\tiny\textcircled{d}}(1-p)][(1-p)b(1-p)]^{m+1}=(1-p)b^{\tiny\textcircled{d}}b^{m+1}.$$
This implies that $$\begin{array}{rl}
&||[(1-p)b(1-p)]^m-[(1-p)b^{\tiny\textcircled{d}}(1-p)][(1-p)b(1-p)]^{m+1}||^{\frac{1}{m}}\\
\leq &||1-p||^{\frac{1}{m}}||b^m-b^{\tiny\textcircled{d}}b^{m+1}||^{\frac{1}{m}}.
\end{array}$$ Therefore $$\lim_{m\to \infty}||[(1-p)b(1-p)]^m-[(1-p)b^{\tiny\textcircled{d}}(1-p)][(1-p)b(1-p)]^{m+1}||^{\frac{1}{m}}=0.$$
Accordingly, $(1-aa^{\tiny\textcircled{d}})b(1-aa^{\tiny\textcircled{d}})\in \mathcal{A}^{\tiny\textcircled{d}}$, as asserted.\end{proof}

We have at our disposal all the information necessary to prove the following.

\begin{thm} Let $a,b\in \mathcal{A}^{\tiny\textcircled{d}}$. Then the following are equivalent:\end{thm}
\begin{enumerate}
\item [(1)] $a\leq^{{\tiny\textcircled{d}}}b$.
\vspace{-.5mm}
\item [(2)] There exist $e_1=aa^{\tiny\textcircled{d}}$ and $e_2=e_2^*=e_2^2$ such that
$$a=\left(
\begin{array}{ccc}
t_1&s_1&s_2\\
0&n_1&n_2\\
0&n_3&n_4
\end{array}
\right)_{e\times e}, b=\left(
\begin{array}{ccc}
t_1&s_1&s_2\\
0&t_3&t_4\\
0&0&t_5
\end{array}
\right)_{e\times e},$$ where $e=\{ e_1,e_2,1-e_1-e_2\}, t_1\in (e_1\mathcal{A}e_1)^{-1}, t_3\in (e_2\mathcal{A}e_2)^{-1}$ and $n_1+n_2+n_3+n_4,t_5\in
\mathcal{A}^{qnil}.$
\end{enumerate}
\begin{proof} $(1)\Rightarrow (2)$ In view of Lemma 4.2, we have
$$a=\left(
\begin{array}{cc}
t_1&t_2\\
0&a_2
\end{array}
\right)_p,b=\left(
\begin{array}{cc}
t_1&t_2\\
0&c
\end{array}
\right)_p,$$ where $t_1\in (p\mathcal{A}p)^{-1}, a_2\in [(1-p)\mathcal{A}(1-p)]^{qnil}, c\in (1-p)\mathcal{A}(1-p).$ By using Lemma 4.3,
$c\in \mathcal{A}^{\tiny\textcircled{d}}$. Let $q=cc^{\tiny\textcircled{d}}$. In light of Lemma 4.1, we have
$$c=\left(
\begin{array}{cc}
t_3&t_4\\
0&t_5
\end{array}
\right)_{q\times q},$$ where $t_3\in (q\mathcal{A}q)^{-1}, t_5\in ((1-q)\mathcal{A}(1-q))^{qnil}$.

We easily check that $pq=0=qp$. Let $e_1=p, e_2=q$ and $e_3=1-p-q$. Then $e_1+e_2+e_3=1$ and $e_ie_j=0 (i\neq j)$. Set $e=\{ e_1,e_2,e_3\}$.
Let $t_1=e_1be_1, s_1=e_1be_2, s_2=e_1be_3, t_3=e_2be_2, t_4=e_2be_3,t_5=e_3be_3.$

Claim 1. $b=\left(
\begin{array}{ccc}
t_1&s_1&s_2\\
0&t_3&t_4\\
0&0&t_5
\end{array}
\right)_{e\times e}.$ This is obvious as $t_3=e_2b_3e_2, t_4=e_2b_3e_3,t_5=e_3b_3e_3.$

Claim 2. $a=\left(
\begin{array}{ccc}
t_1&s_1&s_2\\
0&n_1&n_2\\
0&n_3&n_4
\end{array}
\right)_{e\times e}$. Clearly, we have
$$\begin{array}{rll}
e_1ae_1&=&pap=a(a^{\tiny\textcircled{d}}a)e_1=a(a^{\tiny\textcircled{d}}b)e_1=e_1be_1=t_1,\\
e_1ae_2&=&a(a^{\tiny\textcircled{d}}a)e_2=a(a^{\tiny\textcircled{d}}b)e_2=s_1,\\
e_1ae_3&=&a(a^{\tiny\textcircled{d}}a)e_3=a(a^{\tiny\textcircled{d}}b)e_3=s_2.
\end{array}$$
Moreover, we get $n_1+n_2+n_3+n_4=a_2\in \mathcal{A}^{qnil},$ as desired.

$(2)\Rightarrow (1)$ Since $e_1a^{\tiny\textcircled{d}}e_1=a^{\tiny\textcircled{d}}$, we see that
$$a^{\tiny\textcircled{d}}=\left(
\begin{array}{ccc}
a^{\tiny\textcircled{d}}&0&0\\
0&0&0\\
0&0&0
\end{array}
\right)_{e\times e}.$$ By directly computing, we have $$\begin{array}{rll}
aa^{\tiny\textcircled{d}}&=&\left(
\begin{array}{ccc}
t_1&s_1&s_2\\
0&n_1&n_2\\
0&n_3&n_4
\end{array}
\right)_{e\times e}\left(
\begin{array}{ccc}
a^{\tiny\textcircled{d}}&0&0\\
0&0&0\\
0&0&0
\end{array}
\right)_{e\times e},\\
ba^{\tiny\textcircled{d}}&=&\left(
\begin{array}{ccc}
t_1&s_1&s_2\\
0&t_3&t_4\\
0&0&t_5
\end{array}
\right)_{e\times e}\left(
\begin{array}{ccc}
a^{\tiny\textcircled{d}}&0&0\\
0&0&0\\
0&0&0
\end{array}
\right)_{e\times e}.
\end{array}$$ Thus, $$aa^{\tiny\textcircled{d}}=\left(
\begin{array}{ccc}
t_1a^{\tiny\textcircled{d}}&0&0\\
0&0&0\\
0&0&0
\end{array}
\right)_{e\times e}=ba^{\tiny\textcircled{d}}.$$ Likewise, we prove that $a^{{\tiny\textcircled{d}}}a=a^{{\tiny\textcircled{d}}}b$. Therefore
$a\leq^{{\tiny\textcircled{d}}}b$.\end{proof}

In~\cite[Theorem 2]{D3} and ~\cite[Theorem 4.2]{W}, the core-EP order of two complex matrices were investigated. As an immediate consequence of Theorem 4.4, we derive an alternative new characterization of core-EP order for complex matrices.

\begin{cor} Let $A,B\in {\Bbb C}^{n\times n}$. Then the following are equivalent:\end{cor}
\begin{enumerate}
\item [(1)] $A\leq^{{\tiny\textcircled{D}}}B$.
\vspace{-.5mm}
\item [(2)] There exist $E_1=AA^{\tiny\textcircled{D}}$ and $E_2=E_2^*=E_2^2$ such that
$$A=\left(
\begin{array}{ccc}
T_1&S_1&S_2\\
0&N_1&N_2\\
0&N_3&N_4
\end{array}
\right)_{E\times E}, B=\left(
\begin{array}{ccc}
T_1&S_1&S_2\\
0&T_3&T_4\\
0&0&T_5
\end{array}
\right)_{E\times E},$$
\end{enumerate} where $E=\{ E_1,E_2,I_n-E_1-E_2\}, T_1\in (E_1{\Bbb C}^{n\times n}E_1)^{-1}, T_3\in (E_2{\Bbb C}^{n\times n}E_2)^{-1}$ and
$N_1+N_2+N_3+N_4,T_5\in {\Bbb C}^{n\times n}$ are nilpotent.

\end{document}